\newcommand{\Year}[1]{\gdef\@Year{#1}}
\newcommand{\Volume}[1]{\gdef\@Volume{#1}}
\newcommand{\Issue}[1]{\gdef\@Issue{#1}}

\def \Idates {\let\thefootnote\relax\footnotetext{Received \DTMusedate{ReceiptDate}; accepted \DTMusedate{AcceptanceDate}; published \DTMusedate{PublicationDate}.}}
\def\publname{Journal of Math-Recherche and Applications \newline
  ISSN 2291-8639 \newline
  Volume \@2018, Number \@Issue \ (\@2018), \thepage-\pageref{LastPage} \newline
  http://revues.imist.ma}
\documentclass[psamsfonts]{amsart}

\usepackage{amssymb,amsfonts}
\usepackage[all,arc]{xy}
\usepackage{enumerate}
\usepackage{mathrsfs}

\newtheorem{thm}{Theorem}[section]
\newtheorem{cor}[thm]{Corollary}

\newtheorem{lem}[thm]{Lemma}

\theoremstyle{definition}
\newtheorem{defn}[thm]{Definition}

\newtheorem{exmp}[thm]{Example}

\theoremstyle{remark}

\makeatletter
\let\c@equation\c@thm
\makeatother
\numberwithin{equation}{section}

\title{Perturbation and Stability of Operator Frame for $End_{\mathcal{A}}^{\ast}(\mathcal{H})$}


\author[Mohamed Rossafi and Abdellatif Akhlidj]{\bfseries Mohamed Rossafi$^{1,*}$ and Abdellatif Akhlidj$^2$}

\address{$^1$Department of Mathematics\\
	University of Ibn Tofail\\
	B.P. 133, Kenitra \\
	Morocco}
\address{$^2$Department of Mathematics\\
	University of Hassan II\\
	Casablanca \\
	Morocco}
\address{$^*$Corresponding author: \textnormal{rossafimohamed@gmail.com}}

\date{published .. .., 2018}

\begin{document}

\begin{abstract}

Frame theory is recently an active research area in mathematics, computer science, and engineering with many exciting applications in a variety of different fields. In this paper, we firstly give a characterization of operator frame for $End_{\mathcal{A}}^{\ast}(\mathcal{H})$ and $K$-operator frames for $End_{\mathcal{A}}^{\ast}(\mathcal{H})$. Lastly we consider the stability of operator frame for $End_{\mathcal{A}}^{\ast}(\mathcal{H})$ and $K$-operator frames for $End_{\mathcal{A}}^{\ast}(\mathcal{H})$ under perturbation and we establish some results.

\end{abstract}

\maketitle




\section{Introduction and preliminaries}

Frames for Hilbert spaces were introduced in 1952 by Duffin and Schaefer \cite{Duf}. They abstracted the fundamental notion of Gabor \cite{Gab} to study signal processing. Many generalizations of frames were introduced, frames of subspaces \cite{Asg}, Pseudo-frames \cite{Oga}, oblique frames \cite{Eld}, g-frames \cite{Kho2}, $\ast$-frame \cite{Deh} and $\ast$-K-g-frames \cite{Ros2} in Hilbert $ \mathcal{A} $-modules. In 2000, Frank-Larson \cite{Lar1} introduced the notion of frames in Hilbert $ \mathcal{A} $-modules as a generalization of frames in Hilbert spaces. Recentely, A. Khosravi and B. Khosravi \cite{Kho2} introduced the g-frames theory in Hilbert $ \mathcal{A} $-modules, and Alijani, and Dehghan \cite{Deh} introduced the g-frames theories in Hilbert $ \mathcal{A} $-modules. N. Bounader and S. Kabbaj \cite{Kab} and A. Alijani \cite{Ali} introduced the $\ast$-g-frames which are generalizations of g-frames in Hilbert $ \mathcal{A} $-modules. The notion of Operator frame for the space $B(\mathcal{H} )$ of all bounded linear operators on Hilbert space $\mathcal{H}$ was introduced by Chun-Yan Li and Huai-Xin Cao \cite{Li} and the notion of $K$-operator frames on $B(\mathcal{H})$ for an operator $K\in B(\mathcal{H})$ was introduced by C. Shekhar and S. K. Kaushik \cite{She}, M. Rossafi and S. Kabbaj \cite{Ros} introduced the notion of Operator frame for the space $End_{\mathcal{A}}^{\ast}(\mathcal{H})$ of all adjointable operators on a Hilbert $\mathcal{A}$-module $\mathcal{H}$ and the notion of $K$-Operator frame on $End_{\mathcal{A}}^{\ast}(\mathcal{H})$ for an operator $K\in End_{\mathcal{A}}^{\ast}(\mathcal{H})$. In this paper, we give a characterization of operator frame for $End_{\mathcal{A}}^{\ast}(\mathcal{H})$ and $K$-operator frames for $End_{\mathcal{A}}^{\ast}(\mathcal{H})$. Also, we consider the stability of operator frame for $End_{\mathcal{A}}^{\ast}(\mathcal{H})$ and $K$-operator frames for $End_{\mathcal{A}}^{\ast}(\mathcal{H})$ under perturbation and we establish some results of the finite sum of operator frame for $End_{\mathcal{A}}^{\ast}(\mathcal{H})$ and $K$-operator frames for $End_{\mathcal{A}}^{\ast}(\mathcal{H})$.\\
Let $\mathbb{J}$ be a finite or countable index subset of $\mathbb{N}$. In this section we briefly recall the definitions and basic properties of $C^{\ast}$-algebra, Hilbert $\mathcal{A}$-modules, frame in Hilbert $\mathcal{A}$-modules. For information about frames in Hilbert spaces we refer to \cite{Chr}. Our reference for $C^{\ast}$-algebras is \cite{{Dav},{Con}}. For a $C^{\ast}$-algebra $\mathcal{A}$ if $a\in\mathcal{A}$ is positive we write $a\geq 0$ and $\mathcal{A}^{+}$ denotes the set of positive elements of $\mathcal{A}$.
\begin{defn}\cite{Con}.
	If $\mathcal{A}$ is a Banach algebra, an involution is a map $ a\rightarrow a^{\ast} $ of $\mathcal{A}$ into itself such that for all $a$ and $b$ in $\mathcal{A}$ and all scalars $\alpha$ the following conditions hold:
	\begin{enumerate}
		\item  $(a^{\ast})^{\ast}=a$.
		\item  $(ab)^{\ast}=b^{\ast}a^{\ast}$.
		\item  $(\alpha a+b)^{\ast}=\bar{\alpha}a^{\ast}+b^{\ast}$.
	\end{enumerate}
\end{defn}
\begin{defn}\cite{Con}.
	A $\mathcal{C}^{\ast}$-algebra $\mathcal{A}$ is a Banach algebra with involution such that :$$\|a^{\ast}a\|=\|a\|^{2}$$ for every $a$ in $\mathcal{A}$.
\end{defn}
\begin{exmp}
	$ \mathcal{B}=B(\mathcal{H}) $	the algebra of bounded operators on a Hilbert space, is a  $\mathcal{C}^{\ast}$-algebra, where for each operator $A$, $A^{\ast}$ is the adjoint of $A$.
\end{exmp}
\begin{defn}\cite{Kap}.
	Let $ \mathcal{A} $ be a unital $C^{\ast}$-algebra and $\mathcal{H}$ be a left $ \mathcal{A} $-module, such that the linear structures of $\mathcal{A}$ and $ \mathcal{H} $ are compatible. $\mathcal{H}$ is a pre-Hilbert $\mathcal{A}$-module if $\mathcal{H}$ is equipped with an $\mathcal{A}$-valued inner product $\langle.,.\rangle :\mathcal{H}\times\mathcal{H}\rightarrow\mathcal{A}$, such that is sesquilinear, positive definite and respects the module action. In the other words,
	\begin{itemize}
		\item [(i)] $ \langle x,x\rangle\geq0 $ for all $ x\in\mathcal{H} $ and $ \langle x,x\rangle=0$ if and only if $x=0$.
		\item [(ii)] $\langle ax+y,z\rangle=a\langle x,y\rangle+\langle y,z\rangle$ for all $a\in\mathcal{A}$ and $x,y,z\in\mathcal{H}$.
		\item[(iii)] $ \langle x,y\rangle=\langle y,x\rangle^{\ast} $ for all $x,y\in\mathcal{H}$.
	\end{itemize}	 
\end{defn}
For $x\in\mathcal{H}, $ we define $||x||=||\langle x,x\rangle||^{\frac{1}{2}}$. If $\mathcal{H}$ is complete with $||.||$, it is called a Hilbert $\mathcal{A}$-module or a Hilbert $C^{\ast}$-module over $\mathcal{A}$. For every $a$ in a $C^{\ast}$-algebra $\mathcal{A}$, we have $|a|=(a^{\ast}a)^{\frac{1}{2}}$ and the $\mathcal{A}$-valued norm on $\mathcal{H}$ is defined by $|x|=\langle x, x\rangle^{\frac{1}{2}}$ for $x\in\mathcal{H}$.
\begin{defn}\cite{Lar1}.
	Let $ \mathcal{H} $ be a Hilbert $\mathcal{A}$-module. A family $\{x_{i}\}_{i\in\mathbb{J}}$ of elements of $\mathcal{H}$ is a frame for $ \mathcal{H} $, if there
	exist two positive constants $A$ , $B$, such that for all $x\in\mathcal{H}$,
	\begin{equation}\label{eq1}
	A\langle x,x\rangle\leq\sum_{i\in\mathbb{J}}\langle x,x_{i}\rangle\langle x_{i},x\rangle\leq B\langle x,x\rangle.
	\end{equation}
	The numbers $A$ and $B$ are called lower and upper bound of the frame, respectively. If $A=B=\lambda$, the frame is $\lambda$-tight. If $A = B = 1$, it is called a normalized tight frame or a Parseval frame. If the sum in the middle of \eqref{eq1} is convergent in norm, the frame is called standard. If only upper inequality of \eqref{eq1} hold, then $\{x_{i}\}_{i\in I}$ is called a Bessel sequence for $\mathcal{H}$.
\end{defn}
The following lemmas will be used to prove our mains results
\begin{lem}\label{l1}\cite{Pas}.
	Let $\mathcal{H}$ be a Hilbert $\mathcal{A}$-module. If $T\in End_{\mathcal{A}}^{\ast}(\mathcal{H})$, then $$\langle Tx,Tx\rangle\leq\|T\|^{2}\langle x,x\rangle, \forall x\in\mathcal{H}$$
\end{lem}
\begin{lem}\label{l2}\cite{Zha}.
	Let $\mathcal{H}$ be a Hilbert $\mathcal{A}$-module over a $C^{\ast}$-algebra $\mathcal{A}$. Let $T, S\in End_{\mathcal{A}}^{\ast}(\mathcal{H})$. If $Rang(S)$ is closed, then the following statements are equivalent:
	\begin{itemize}
		\item [(i)] $Rang(T)\subseteq Rang(S)$.
		\item [(ii)] $\lambda TT^{\ast}\leq SS^{\ast}$ for some $\lambda>0$.
		\item [(iii)] There exists a positive real number $\mu>0$ such that $\mu\|T^{\ast}x\|^{2}\leq\|S^{\ast}x\|^{2}$, for all $x\in\mathcal{H}$.
		\item[(iv)] There exists $Q\in End_{\mathcal{A}}^{\ast}(\mathcal{H})$ such that $T=SQ$.
	\end{itemize}
\end{lem}
\begin{lem}\label{l3}\cite{Ara}.
	Let $\mathcal{H}$ be a Hilbert $\mathcal{A}$-module over a $C^{\ast}$-algebra $\mathcal{A}$, and $T\in End_{\mathcal{A}}^{\ast}(\mathcal{H})$ such that $T^{\ast}=T$. The following statements are equivalent:
	\begin{itemize}
		\item [(i)] $T$ is surjective.
		\item[(ii)] There are $m, M>0$ such that $m\|x\|\leq\|Tx\|\leq M\|x\|$, for all $x\in\mathcal{H}$.
		\item[(iii)] There are $m', M'>0$ such that $m'\langle x,x\rangle\leq\langle Tx,Tx\rangle\leq M'\langle x,x\rangle$ for all $x\in\mathcal{H}$.
	\end{itemize}
\end{lem}
\begin{lem}\label{l4}\cite{Deh}.
	Let $\mathcal{H}$ and $\mathcal{K}$ are two Hilbert $\mathcal{A}$-modules and $T\in End^{\ast}(\mathcal{H},\mathcal{K})$. Then:
	\begin{itemize}
		\item [(i)] If $T$ is injective and $T$ has closed range, then the adjointable map $T^{\ast}T$ is invertible and $$\|(T^{\ast}T)^{-1}\|^{-1}\leq T^{\ast}T\leq\|T\|^{2}.$$
		\item  [(ii)]	If $T$ is surjective, then the adjointable map $TT^{\ast}$ is invertible and $$\|(TT^{\ast})^{-1}\|^{-1}\leq TT^{\ast}\leq\|T\|^{2}.$$
	\end{itemize}	
\end{lem}
\begin{defn}
	A sequence $\{\alpha_{i}\}_{i\in\mathbb{J}}\subset\mathbb{R}$ is said to be positively confined if: $$0<\inf_{\mathbb{J}}\alpha_{i}\leq\sup_{\mathbb{J}}\alpha_{i}<\infty.$$
\end{defn}

\section{Characterization of operator frame for $End_{\mathcal{A}}^{\ast}(\mathcal{H})$}
In this section, we give a characterazation of operator frame for $End_{\mathcal{A}}^{\ast}(\mathcal{H})$.
\begin{defn} \cite{Ros}.
	A family of adjointable operators $\{T_{i}\}_{i\in\mathbb{J}}$ on a Hilbert $\mathcal{A}$-module $\mathcal{H}$ is said to be an operator frame for $End_{\mathcal{A}}^{\ast}(\mathcal{H})$, if there exists positive constants $A, B > 0$ such that 
	\begin{equation}\label{eq3}
	A\langle x,x\rangle\leq\sum_{i\in\mathbb{J}}\langle T_{i}x,T_{i}x\rangle\leq B\langle x,x\rangle, \forall x\in\mathcal{H}.
	\end{equation}
	The numbers $A$ and $B$ are called lower and upper bound of the operator frame, respectively. If $A=B=\lambda$, the operator frame is $\lambda$-tight. If $A = B = 1$, it is called a normalized tight operator frame or a Parseval operator frame. If only upper inequality of \eqref{eq3} hold, then $\{T_{i}\}_{i\in\mathbb{J}}$ is called an operator Bessel sequence for $End_{\mathcal{A}}^{\ast}(\mathcal{H})$. The operator frame is standard if for every $x\in\mathcal{H}$, the sum in \eqref{eq3} converges in norm.
\end{defn}
Let $\{T_{i}\}_{i\in\mathbb{J}}$ be an operator frame for $End_{\mathcal{A}}^{\ast}(\mathcal{H})$. Define an operator 
$$R:\mathcal{H}\rightarrow l^{2}(\mathcal{H})\;\;by\;\; Rx=\{T_{i}x\}_{i\in\mathbb{J}}, \forall x\in\mathcal{H}.$$
The operator $R$ is called the analysis operator of the operator frame $ \{T_{i}\}_{i\in\mathbb{J}} $.\\
The adjoint of the  analysis operator $R$, $$R^{\ast}(\{x_{i}\}_{i\in\mathbb{J}}):l^{2}(\mathcal{H})\rightarrow\mathcal{H}$$ is defined by $$R^{\ast}(\{x_{i}\}_{i\in\mathbb{J}})=\sum_{i\in\mathbb{J}}T_{i}^{\ast}x_{i}, \forall\{x_{i}\}_{i\in\mathbb{J}}\in l^{2}(\mathcal{H}).$$
The operator $R^{\ast}$ is called the synthesis operator of the operator frame $ \{T_{i}\}_{i\in\mathbb{J}} $.\\
By composing $R$ and $R^{\ast}$, the frame operator $S_{T}:\mathcal{H}\rightarrow\mathcal{H}$ for the operator frame is given by $$	S_{T}(x)=R^{\ast}Rx=\sum_{i\in\mathbb{J}}T_{i}^{\ast}T_{i}x$$
\begin{thm}
	Let $\{T_{i}\}_{i\in\mathbb{J}}$ be a family of adjointable operators  on a Hilbert $\mathcal{A}$-module $\mathcal{H}$. Suppose that
	$\sum_{i\in\mathbb{J}}\langle T_{i}x,T_{i}x\rangle$ converge in norm for all $x\in\mathcal{H}$. Then $\{T_{i}\}_{i\in\mathbb{J}}$ is an operator frame for $End_{\mathcal{A}}^{\ast}(\mathcal{H})$ if and only if there exist positive constants $A, B > 0$ such that
	\begin{equation}\label{equa2}
	A\|x\|^{2}\leq\|\sum_{i\in\mathbb{J}}\langle T_{i}x,T_{i}x\rangle\|\leq B\|x\|^{2}, \forall x\in\mathcal{H}.
	\end{equation} 
\end{thm}
\begin{proof}
	Suppose that $\{T_{i}\}_{i\in\mathbb{J}}$ is an operator frame for $End_{\mathcal{A}}^{\ast}(\mathcal{H})$. Since for any $x\in\mathcal{H}$, there is $\langle x,x\rangle\geq0$, combined with the definition of an operator frame we know that \eqref{equa2} holds.\\
	Now suppose that \eqref{equa2} holds. We know that the frame operator $S_{T}$ is positive, self-adjoint and inversible, hence $$\langle S_{T}^{\frac{1}{2}}x,S_{T}^{\frac{1}{2}}x\rangle=\langle S_{T}x,x\rangle=\sum_{i\in\mathbb{J}}\langle T_{i}x,T_{i}x\rangle.$$
	So we have $\sqrt{A}\|x\|\leq\|S_{T}^{\frac{1}{2}}x\|\leq\sqrt{B}\|x\|$ for any $x\in\mathcal{H}$. According to Lemma \ref{l3}, there are constants $m, M>0$ such that
	$$m\langle x,x\rangle\leq\langle S_{T}^{\frac{1}{2}}x,S_{T}^{\frac{1}{2}}x\rangle=\sum_{i\in\mathbb{J}}\langle T_{i}x,T_{i}x\rangle\leq M\langle x,x\rangle,$$
	which implies that $\{T_{i}\}_{i\in\mathbb{J}}$ is an operator frame for $End_{\mathcal{A}}^{\ast}(\mathcal{H})$.
\end{proof}

\section{perturbation and stability of operator frames for $End_{\mathcal{A}}^{\ast}(\mathcal{H})$}
The theory of perturbation is a very important tool in many area of applied mathematics. In this section we study the stability of operator frame under small perturbations. we begin with the following theorems.
\begin{thm}
	Let $\{T_{i}\}_{i\in\mathbb{J}}$ be an operator frame for $End_{\mathcal{A}}^{\ast}(\mathcal{H})$ with bound $A$ and $B$. If $\{R_{i}\}_{i\in\mathbb{J}}\in End_{\mathcal{A}}^{\ast}(\mathcal{H})$ is an operator Bessel sequence with a bound $M<A$, then $\{T_{i}\pm R_{i}\}_{i\in\mathbb{J}}$ is an operator frame for $End_{\mathcal{A}}^{\ast}(\mathcal{H})$.
\end{thm}
\begin{proof}
	We only prove the case that $\{T_{i}+R_{i}\}_{i\in\mathbb{J}}$ is an operator frame for $End_{\mathcal{A}}^{\ast}(\mathcal{H})$. The other case is similar.\\
	For any $x\in\mathcal{H}$, we have
	\begin{align}\label{31}
	\|\sum_{i\in\mathbb{J}}\langle(T_{i}+R_{i})x,(T_{i}+R_{i})x\rangle\|^{\frac{1}{2}}&=\|\{(T_{i}+R_{i})x\}_{i\in\mathbb{J}}\|\leq\|\{T_{i}x\}_{i\in\mathbb{J}}\|+\|\{R_{i}x\}_{i\in\mathbb{J}}\|\notag\\
	&=\|\sum_{i\in\mathbb{J}}\langle T_{i}x,T_{i}x\rangle\|^{\frac{1}{2}}+\|\sum_{i\in\mathbb{J}}\langle R_{i}x,R_{i}x\rangle\|^{\frac{1}{2}}\notag\\
	&\leq\sqrt{B}\|x\|+\sqrt{M}\|x\|=(\sqrt{B}+\sqrt{M})\|x\|.
	\end{align}
	and
	\begin{align}\label{32}
	\|\sum_{i\in\mathbb{J}}\langle(T_{i}+R_{i})x,(T_{i}+R_{i})x\rangle\|^{\frac{1}{2}}&=\|\{(T_{i}+R_{i})x\}_{i\in\mathbb{J}}\|\geq\|\{T_{i}x\}_{i\in\mathbb{J}}\|-\|\{R_{i}x\}_{i\in\mathbb{J}}\|\notag\\
	&=\|\sum_{i\in\mathbb{J}}\langle T_{i}x,T_{i}x\rangle\|^{\frac{1}{2}}-\|\sum_{i\in\mathbb{J}}\langle R_{i}x,R_{i}x\rangle\|^{\frac{1}{2}}\notag\\
	&\geq\sqrt{A}\|x\|-\sqrt{M}\|x\|=(\sqrt{A}-\sqrt{M})\|x\|.
	\end{align}
	Of \eqref{31} and \eqref{32} we obtain
	$$(\sqrt{A}-\sqrt{M})^{2}\|x\|^{2}\leq\|\sum_{i\in\mathbb{J}}\langle(T_{i}+R_{i})x,(T_{i}+R_{i})x\rangle\|\leq(\sqrt{B}+\sqrt{M})^{2}\|x\|^{2}.$$
	so we have $\{T_{i}+R_{i}\}_{i\in\mathbb{J}}$ is an operator frame for $End_{\mathcal{A}}^{\ast}(\mathcal{H})$.
\end{proof}
\begin{thm}
	Let $\{T_{i}\}_{i\in\mathbb{J}}$ be an operator frame for $End_{\mathcal{A}}^{\ast}(\mathcal{H})$ with bound $A$ and $B$. Let $\{R_{i}\}_{i\in\mathbb{J}}\in End_{\mathcal{A}}^{\ast}(\mathcal{H})$. Then the following statements are equivalent:
	\begin{itemize}
		\item [(i)] $\{R_{i}\}_{i\in\mathbb{J}}$ is an operator frame for $End_{\mathcal{A}}^{\ast}(\mathcal{H})$.
		\item[(ii)] There exists a constant $M>0$, such that for all $x\in\mathcal{H}$, we have
		\begin{equation}\label{eq33}
		\|\sum_{i\in\mathbb{J}}\langle(T_{i}-R_{i})x,(T_{i}-R_{i})x\rangle\|\leq M\min(\|\sum_{i\in\mathbb{J}}\langle T_{i}x,T_{i}x\rangle\|,\|\sum_{i\in\mathbb{J}}\langle R_{i}x,R_{i}x\rangle\|).
		\end{equation}
	\end{itemize} 
\end{thm}
\begin{proof}
	First, let $\{R_{i}\}_{i\in\mathbb{J}}\in End_{\mathcal{A}}^{\ast}(\mathcal{H})$ be an operator frame with bound $C$ and $D$. Then for any $x\in\mathcal{H}$, we have
	\begin{align*}
	\|\sum_{i\in\mathbb{J}}\langle(T_{i}-R_{i})x,(T_{i}-R_{i})x\rangle\|^{\frac{1}{2}}&=\|\{(T_{i}-R_{i})x\}_{i\in\mathbb{J}}\|\leq\|\{T_{i}x\}_{i\in\mathbb{J}}\|+\|\{R_{i}x\}_{i\in\mathbb{J}}\|\\
	&=\|\sum_{i\in\mathbb{J}}\langle T_{i}x,T_{i}x\rangle\|^{\frac{1}{2}}+\|\sum_{i\in\mathbb{J}}\langle R_{i}x,R_{i}x\rangle\|^{\frac{1}{2}}\\
	&\leq\|\sum_{i\in\mathbb{J}}\langle T_{i}x,T_{i}x\rangle\|^{\frac{1}{2}}+\sqrt{D}\|x\|\\
	&\leq\|\sum_{i\in\mathbb{J}}\langle T_{i}x,T_{i}x\rangle\|^{\frac{1}{2}}+\sqrt{\frac{D}{A}}\|\sum_{i\in\mathbb{J}}\langle T_{i}x,T_{i}x\rangle\|^{\frac{1}{2}}\\
	&=(1+\sqrt{\frac{D}{A}})\|\sum_{i\in\mathbb{J}}\langle T_{i}x,T_{i}x\rangle\|^{\frac{1}{2}}
	\end{align*}
	Similary we can obtain
	\begin{equation*}
	\|\sum_{i\in\mathbb{J}}\langle(T_{i}-R_{i})x,(T_{i}-R_{i})x\rangle\|^{\frac{1}{2}}\leq(1+\sqrt{\frac{B}{C}})\|\sum_{i\in\mathbb{J}}\langle R_{i}x,R_{i}x\rangle\|^{\frac{1}{2}}
	\end{equation*}
	Let $M=\min\{1+\sqrt{\frac{D}{A}},1+\sqrt{\frac{B}{C}}\}$, then \eqref{eq33} holds.\\
	Next we suppose that \eqref{eq33} holds. For evry $x\in\mathcal{H}$, we have
	\begin{align}\label{eq34}
	\sqrt{A}\|x\|\leq\|\sum_{i\in\mathbb{J}}\langle T_{i}x,T_{i}x\rangle\|^{\frac{1}{2}}&=\|\{T_{i}x\}_{i\in\mathbb{J}}\|\leq\|\{(T_{i}-R_{i})x\}_{i\mathbb{J}}\|+\|\{R_{i}x\}_{i\in\mathbb{J}}\|\notag\\
	&=\|\sum_{i\in\mathbb{J}}\langle(T_{i}-R_{i})x,(T_{i}-R_{i})x\rangle\|^{\frac{1}{2}}+\|\sum_{i\in\mathbb{J}}\langle R_{i}x,R_{i}x\rangle\|^{\frac{1}{2}}\notag\\
	&\leq\sqrt{M}\|\sum_{i\in\mathbb{J}}\langle R_{i}x,R_{i}x\rangle\|^{\frac{1}{2}}+\|\sum_{i\in\mathbb{J}}\langle R_{i}x,R_{i}x\rangle\|^{\frac{1}{2}}\notag\\
	&=(\sqrt{M}+1)\|\sum_{i\in\mathbb{J}}\langle R_{i}x,R_{i}x\rangle\|^{\frac{1}{2}}.
	\end{align}
	Also we have
	\begin{align}\label{eq35}
	\|\sum_{i\in\mathbb{J}}\langle R_{i}x,R_{i}x\rangle\|^{\frac{1}{2}}&=\|\{R_{i}x\}_{i\in\mathbb{J}}\|\leq\|\{T_{i}x\}_{i\in\mathbb{J}}\|+\|\{(T_{i}-R_{i})x\}_{i\mathbb{J}}\|\notag\\
	&=\|\sum_{i\in\mathbb{J}}\langle T_{i}x,T_{i}x\rangle\|^{\frac{1}{2}}+\|\sum_{i\in\mathbb{J}}\langle(T_{i}-R_{i})x,(T_{i}-R_{i})x\rangle\|^{\frac{1}{2}}\notag\\
	&\leq\|\sum_{i\in\mathbb{J}}\langle T_{i}x,T_{i}x\rangle\|^{\frac{1}{2}}+\sqrt{M}\|\sum_{i\in\mathbb{J}}\langle T_{i}x,T_{i}x\rangle\|^{\frac{1}{2}}\notag\\
	&=(\sqrt{M}+1)\sum_{i\in\mathbb{J}}\langle T_{i}x,T_{i}x\rangle\|^{\frac{1}{2}}\notag\\
	&\leq\sqrt{B}(\sqrt{M}+1)\|x\|.
	\end{align}
	Of \eqref{eq34} and \eqref{eq35} we obtain
	\begin{equation*}
	\frac{A}{(\sqrt{M}+1)^{2}}\|x\|^{2}\leq\|\sum_{i\in\mathbb{J}}\langle R_{i}x,R_{i}x\rangle\|\leq B(\sqrt{M}+1)^{2}\|x\|^{2}.
	\end{equation*}
	So we have $\{R_{i}\}_{i\in\mathbb{J}}$ is an operator frame for $End_{\mathcal{A}}^{\ast}(\mathcal{H})$.
\end{proof}
The following theorem give a necessary and sufficient condition for the finite sum of operator frames to be an operator frame.
\begin{thm}
	Let $\{T_{n,i}\}_{i}\subset End_{\mathcal{A}}^{\ast}(\mathcal{H})$, $n=1,2,..,k$, be operator frame for $End_{\mathcal{A}}^{\ast}(\mathcal{H})$ with bound $A_{n}$ and $B_{n}$, let
	$\{\alpha_{n}\}_{n=1}^{k}$ be any scalars. If there exists a constant $\lambda>0$ and some $p\in\{1,2,...,k\}$ such that
	\begin{equation*}
	\lambda\|\{T_{p,i}x\}_{i}\|\leq\|\{\sum_{n=1}^{k}\alpha_{n}T_{n,i}(x)\}_{i}\|, x\in\mathcal{H}.
	\end{equation*}
	Then $\{\sum_{n=1}^{k}\alpha_{n}T_{n,i}\}_{i}$ is an operator frame for $End_{\mathcal{A}}^{\ast}(\mathcal{H})$. And conversely.
\end{thm}
\begin{proof}
	For any $x\in\mathcal{H}$, we have
	\begin{align*}
	\sqrt{A_{p}}\lambda\|\langle x,x\rangle\|^{\frac{1}{2}}&\leq\lambda\|\{T_{p,i}x\}_{i}\|\\
	&\leq\|\{\sum_{n=1}^{k}\alpha_{n}T_{n,i}(x)\}_{i}\|\\
	&\leq\sum_{n=1}^{k}|\alpha_{n}|\|\{T_{n,i}x\}_{i}\|\\
	&\leq\max_{1\leq n\leq k}|\alpha_{n}|\sum_{n=1}^{k}\|\{T_{n,i}x\}_{i}\|\\
	&\leq\max_{1\leq n\leq k}|\alpha_{n}|(\sum_{n=1}^{k}\sqrt{B_{n}})\|\langle x,x\rangle\|^{\frac{1}{2}}.
	\end{align*}
	Then forall $x\in\mathcal{H}$, we have
	\begin{equation*}
	\sqrt{A_{p}}\lambda\|\langle x,x\rangle\|^{\frac{1}{2}}\leq\|\{\sum_{n=1}^{k}\alpha_{n}T_{n,i}(x)\}_{i}\|\leq\max_{1\leq n\leq k}|\alpha_{n}|(\sum_{n=1}^{k}\sqrt{B_{n}})\|\langle x,x\rangle\|^{\frac{1}{2}}.
	\end{equation*}
	So forall $x\in\mathcal{H}$,
	\begin{equation*}
	A_{p}\lambda^{2}\|\langle x,x\rangle\|\leq\|\{\sum_{n=1}^{k}\alpha_{n}T_{n,i}(x)\}_{i}\|^{2}\leq(\max_{1\leq n\leq k}|\alpha_{n}|)^{2}(\sum_{n=1}^{k}\sqrt{B_{n}})^{2}\|\langle x,x\rangle\|.
	\end{equation*}
	Hence $\{\sum_{n=1}^{k}\alpha_{n}T_{n,i}\}_{i}$ is an operator frame for $End_{\mathcal{A}}^{\ast}(\mathcal{H}).$\\
	Conversly, let $\{\sum_{n=1}^{k}\alpha_{n}T_{n,i}\}_{i}$ is an operator frame for $End_{\mathcal{A}}^{\ast}(\mathcal{H})$ with bounds $A$, $B$ and let for any $p\in\{1,2,...,k\}$, $\{T_{p,i}\}_{i}$ be an  an operator frame for $End_{\mathcal{A}}^{\ast}(\mathcal{H})$ with bounds $A_{p}$ and $B_{p}$. Then, for any $x\in\mathcal{H}$, $p\in\{1,2,...,k\}$, we have
	\begin{equation*}
	A_{p}\|\langle x,x\rangle\|\leq\|\{T_{p,i}x\}_{i}\|^{2}\leq B_{p}\|\langle x,x\rangle\|.
	\end{equation*}
	This gives
	\begin{equation*}
	\frac{1}{B_{p}}\|\{T_{p,i}x\}_{i}\|^{2}\leq\|\langle x,x\rangle\|, x\in\mathcal{H}.
	\end{equation*}
	Also, we have
	\begin{equation*}
	A\|\langle x,x\rangle\|\leq\|\{\sum_{n=1}^{k}\alpha_{n}T_{n,i}(x)\}_{i}\|^{2}\leq B\|\langle x,x\rangle\|, x\in\mathcal{H}.
	\end{equation*}
	So,
	\begin{equation*}
	\|\langle x,x\rangle\|\leq\frac{1}{A}\|\{\sum_{n=1}^{k}\alpha_{n}T_{n,i}(x)\}_{i}\|^{2}, x\in\mathcal{H}.
	\end{equation*}
	Hence
	\begin{equation*}
	\frac{A}{B_{p}}\|\{T_{p,i}x\}_{i}\|^{2}\leq\|\{\sum_{n=1}^{k}\alpha_{n}T_{n,i}(x)\}_{i}\|^{2}, x\in\mathcal{H}.
	\end{equation*}
	Then for $\lambda=\frac{A}{B_{p}}$, we have
	\begin{equation*}
	\lambda\|\{T_{p,i}x\}_{i}\|^{2}\leq\|\{\sum_{n=1}^{k}\alpha_{n}T_{n,i}(x)\}_{i}\|^{2}, x\in\mathcal{H}.
	\end{equation*}	
\end{proof}
Next, we give a sufficient condition for the stability of finite sum of operator frames.
\begin{thm}
	For $n=1,2,..,k$, let $\{T_{n,i}\}_{i}\subset End_{\mathcal{A}}^{\ast}(\mathcal{H})$ be an operator frame for $End_{\mathcal{A}}^{\ast}(\mathcal{H})$ with bound $A_{n}$ and $B_{n}$, $\{R_{n,i}\}_{i}\subset End_{\mathcal{A}}^{\ast}(\mathcal{H})$ be any sequence. Let $L:\ell^{2}(\mathcal{H})\rightarrow\ell^{2}(\mathcal{H})$ be a bounded linear operator such that $L(\{\sum_{n=1}^{k}R_{n,i}(x)\}_{i})=\{T_{p,i}(x)\}_{i}$, for some $p\in\{1,2,...,k\}$. If there exists a constant $\lambda>0$ such that
	$$ \|\sum_{i}\langle(T_{n,i}-R_{n,i})x,(T_{n,i}-R_{n,i})x\rangle\|\leq\lambda\|\sum_{i}\langle T_{n,i}x,T_{n,i}x\rangle\|, x\in\mathcal{H}, n=1,2,..,k. $$
	Then $\{\sum_{n=1}^{k}R_{n,i}\}_{i}$ is an operator frame for $End_{\mathcal{A}}^{\ast}(\mathcal{H})$.
\end{thm}
\begin{proof}
	For any $x\in\mathcal{H}$, we have
	\begin{align*}
	\|\{\sum_{n=1}^{k}R_{n,i}(x)\}_{i}\|&\leq\sum_{n=1}^{k}\|\{R_{n,i}(x)\}_{i}\|\\
	&\leq\sum_{n=1}^{k}(\|\{(T_{n,i}-R_{n,i})x\}_{i}\|+\|\{T_{n,i}(x)\}_{i}\|)\\
	&\leq\sum_{n=1}^{k}(\sqrt{\lambda}\|\{T_{n,i}(x)\}_{i}\|+\|\{T_{n,i}(x)\}_{i}\|)\\
	&=(1+\sqrt{\lambda})\sum_{n=1}^{k}\|\{T_{n,i}(x)\}_{i}\|\\
	&\leq(1+\sqrt{\lambda})(\sum_{n=1}^{k}\sqrt{B_{n}})\|\langle x,x\rangle\|^{\frac{1}{2}}.
	\end{align*}
	Also, for each $x\in\mathcal{H}$, we have
	$$\|L(\{\sum_{n=1}^{k}R_{n,i}(x)\}_{i})\|=\|\{T_{p,i}(x)\}_{i}\|$$
	Therefore, we get
	\begin{align*}
	\sqrt{A_{p}}\|\langle x,x\rangle\|^{\frac{1}{2}}\leq\|\{T_{p,i}x\}_{i}\|&=\|L(\{\sum_{n=1}^{k}R_{n,i}(x)\}_{i})\|\\
	&\leq\|L\|\|\{\sum_{n=1}^{k}R_{n,i}(x)\}_{i}\|, x\in\mathcal{H}
	\end{align*}
	This gives
	$$\frac{\sqrt{A_{p}}}{\|L\|}\|\langle x,x\rangle\|^{\frac{1}{2}}\leq\|\{\sum_{n=1}^{k}R_{n,i}(x)\}_{i}\|, x\in\mathcal{H}$$
	Then
	$$\frac{\sqrt{A_{p}}}{\|L\|}\|\langle x,x\rangle\|^{\frac{1}{2}}\leq\|\{\sum_{n=1}^{k}R_{n,i}(x)\}_{i}\|\leq(1+\sqrt{\lambda})(\sum_{n=1}^{k}\sqrt{B_{n}})\|\langle x,x\rangle\|^{\frac{1}{2}},x\in\mathcal{H}.$$
	So
	$$\frac{A_{p}}{\|L\|^{2}}\|\langle x,x\rangle\|\leq\|\{\sum_{n=1}^{k}R_{n,i}(x)\}_{i}\|^{2}\leq(1+\sqrt{\lambda})^{2}(\sum_{n=1}^{k}\sqrt{B_{n}})^{2}\|\langle x,x\rangle\|,x\in\mathcal{H}.$$
	Hence $\{\sum_{n=1}^{k}R_{n,i}\}_{i}$ is an operator frame for $End_{\mathcal{A}}^{\ast}(\mathcal{H})$.
\end{proof}
\section{Characterization of k-operator frames for $End_{\mathcal{A}}^{\ast}(\mathcal{H})$}
We began this section with the following definition.
\begin{defn} \cite{Ros}.
	Let $K\in End_{\mathcal{A}}^{\ast}(\mathcal{H})$. A family of adjointable operators $\{T_{i}\}_{i\in\mathbb{J}}$ on a Hilbert $\mathcal{A}$-module $\mathcal{H}$ is said to be a $K$-operator frame for $End_{\mathcal{A}}^{\ast}(\mathcal{H})$, if there exists positive constants $A, B > 0$ such that 
	\begin{equation}\label{eq6}
	A\langle K^{\ast}x,K^{\ast}x\rangle\leq\sum_{i\in\mathbb{J}}\langle T_{i}x,T_{i}x\rangle\leq B\langle x,x\rangle, \forall x\in\mathcal{H}.
	\end{equation}
	The numbers $A$ and $B$ are called lower and upper bound of the $K$-operator frame, respectively. If $$A\langle K^{\ast}x,K^{\ast}x\rangle=\sum_{i\in\mathbb{J}}\langle T_{i}x,T_{i}x\rangle,$$ the $K$-operator frame is $A$-tight. If $A=1$, it is called a normalized tight $K$-operator frame or a Parseval $K$-operator frame. The $K$-operator frame is standard if for every $x\in\mathcal{H}$, the sum in \eqref{eq6} converges in norm. 
\end{defn}

Let $\{T_{i}\}_{i\in\mathbb{J}}$ be a $K$-operator frame for $End_{\mathcal{A}}^{\ast}(\mathcal{H})$. Define an operator $$R:\mathcal{H}\rightarrow l^{2}(\mathcal{H})\;\; by\;\; Rx=\{T_{i}x\}_{i\in\mathbb{J}}, \forall x\in\mathcal{H}.$$
The operator $R$ is called the analysis operator of the $K$-operator frame $ \{T_{i}\}_{i\in\mathbb{J}} $.\\
The adjoint of the  analysis operator $R$, $$R^{\ast}(\{x_{i}\}_{i\in\mathbb{J}}):l^{2}(\mathcal{H})\rightarrow\mathcal{H}$$ is defined by $$R^{\ast}(\{x_{i}\}_{i\in\mathbb{J}})=\sum_{i\in\mathbb{J}}T_{i}^{\ast}x_{i}, \forall\{x_{i}\}_{i\in\mathbb{J}}\in l^{2}(\mathcal{H}).$$
The operator $R^{\ast}$ is called the synthesis operator of the $K$-operator frame $ \{T_{i}\}_{i\in\mathbb{J}} $.\\
By composing $R$ and $R^{\ast}$, the frame operator $S_{T}:\mathcal{H}\rightarrow\mathcal{H}$ for the $K$-operator frame is given by $$	S_{T}(x)=R^{\ast}Rx=\sum_{i\in\mathbb{J}}T_{i}^{\ast}T_{i}x.$$
\begin{thm}
	Let $\{T_{i}\}_{i\in\mathbb{J}}$ be a family of adjointable operators  on a Hilbert $\mathcal{A}$-module $\mathcal{H}$. Suppose that
	$\sum_{i\in\mathbb{J}}\langle T_{i}x,T_{i}x\rangle$ converge in norm for all $x\in\mathcal{H}$. Then $\{T_{i}\}_{i\in\mathbb{J}}$ is a $K$-operator frame for $End_{\mathcal{A}}^{\ast}(\mathcal{H})$ if and only if there exist positive constants $A, B > 0$ such that
	\begin{equation}\label{eq42}
	A\|K^{\ast}x\|^{2}\leq\|\sum_{i\in\mathbb{J}}\langle T_{i}x,T_{i}x\rangle\|\leq B\|x\|^{2}, \forall x\in\mathcal{H}.
	\end{equation} 
\end{thm}
\begin{proof}
	Suppose that $\{T_{i}\}_{i\in\mathbb{J}}$ is a $K$-operator frame for $End_{\mathcal{A}}^{\ast}(\mathcal{H})$. Since for any $x\in\mathcal{H}$, there is $\langle K^{\ast}x,K^{\ast}x\rangle\geq0$, combined with the definition of a $K$-operator frame we know that \eqref{eq42} holds.\\
	Now suppose that \eqref{eq42} holds. We know that the frame operator $S_{T}$ is positive, self-adjoint and inversible, hence $$\langle S_{T}^{\frac{1}{2}}x,S_{T}^{\frac{1}{2}}x\rangle=\langle S_{T}x,x\rangle=\sum_{i\in\mathbb{J}}\langle T_{i}x,T_{i}x\rangle.$$
	So we have $\sqrt{A}\|K^{\ast}x\|\leq\|S_{T}^{\frac{1}{2}}x\|\leq\sqrt{B}\|x\|$ for any $x\in\mathcal{H}$. According to Lemmas \ref{l2} and \ref{l3}, there are constants $m, M>0$ such that
	$$m\langle K^{\ast}x,K^{\ast}x\rangle\leq\langle S_{T}^{\frac{1}{2}}x,S_{T}^{\frac{1}{2}}x\rangle=\sum_{i\in\mathbb{J}}\langle T_{i}x,T_{i}x\rangle\leq M\langle x,x\rangle,$$
	which implies that $\{T_{i}\}_{i\in\mathbb{J}}$ is a $K$-operator frame for $End_{\mathcal{A}}^{\ast}(\mathcal{H})$.
\end{proof}
\section{perturbation and stability of k-operator frames for $End_{\mathcal{A}}^{\ast}(\mathcal{H})$}
In this section we study the stability of $K$-operator frame under small perturbations. we begin with the following theorems.
\begin{thm}
	Let $\{T_{i}\}_{i\in\mathbb{J}}$ be a $K$-operator frame for $End_{\mathcal{A}}^{\ast}(\mathcal{H})$ with bound $A$ and $B$, let $\{R_{i}\}_{i\in\mathbb{J}}\subset End_{\mathcal{A}}^{\ast}(\mathcal{H})$ and $\{\alpha_{i}\}_{i\in\mathbb{J}},\{\beta_{i}\}_{i\in\mathbb{J}}\subset\mathbb{R}$ be two positively confined sequences. If there exist constants $0\leq\lambda, \mu<1$ such that for all $x\in\mathcal{H}$ we have:
	$$\|\sum_{i\in\mathbb{J}}\langle(\alpha_{i}T_{i}-\beta_{i}R_{i})x,(\alpha_{i}T_{i}-\beta_{i}R_{i})x\rangle\|^{\frac{1}{2}}$$$$\leq\lambda\|\sum_{i\in\mathbb{J}}\langle \alpha_{i}T_{i}x,\alpha_{i}T_{i}x\rangle\|^{\frac{1}{2}}+\mu\|\sum_{i\in\mathbb{J}}\langle \beta_{i}R_{i}x,\beta_{i}R_{i}x\rangle\|^{\frac{1}{2}}.$$
	Then $\{R_{i}\}_{i\in\mathbb{J}}$ is a $K$-operator frame for $End_{\mathcal{A}}^{\ast}(\mathcal{H})$. 
\end{thm}
\begin{proof}
	For any $x\in\mathcal{H}$ we have
	\begin{align*}
	\|\{\beta_{i}R_{i}x\}_{i}\|&\leq\|\{(\alpha_{i}T_{i}-\beta_{i}R_{i})x\}_{i}\|+\|\{\alpha_{i}T_{i}x\}_{i}\|\\
	&\leq\lambda\|\{\alpha_{i}T_{i}x\}_{i}\|+\mu\|\{\beta_{i}R_{i}x\}_{i}\|+\|\{\alpha_{i}T_{i}x\}_{i}\|\\
	&=(1+\lambda)\|\{\alpha_{i}T_{i}x\}_{i}\|+\mu\|\{\beta_{i}R_{i}x\}_{i}\|.
	\end{align*}
	So
	$$(1-\mu)\|\{\beta_{i}R_{i}x\}_{i}\|\leq(1+\lambda)\|\{\alpha_{i}T_{i}x\}_{i}\|.$$
	Hence
	$$(1-\mu)\inf_{\mathbb{J}}\beta_{i}\|\{R_{i}x\}_{i}\|\leq(1+\lambda)\sup_{\mathbb{J}}\alpha_{i}\|\{T_{i}x\}_{i}\|.$$
	Thus
	$$\|\{R_{i}x\}_{i}\|\leq\frac{(1+\lambda)\sup_{\mathbb{J}}\alpha_{i}}{(1-\mu)\inf_{\mathbb{J}}\beta_{i}}\|\{T_{i}x\}_{i}\|.$$
	Also, For any $x\in\mathcal{H}$ we have
	\begin{align*}
	\|\{\alpha_{i}T_{i}x\}_{i}\|&\leq\|\{(\alpha_{i}T_{i}-\beta_{i}R_{i})x\}_{i}\|+\|\{\beta_{i}R{i}x\}_{i}\|\\
	&\leq\lambda\|\{\alpha_{i}T_{i}x\}_{i}\|+\mu\|\{\beta_{i}R_{i}x\}_{i}\|+\|\{\beta_{i}R_{i}x\}_{i}\|\\
	&=\lambda\|\{\alpha_{i}T_{i}x\}_{i}\|+(1+\mu)\|\{\beta_{i}R_{i}x\}_{i}\|.
	\end{align*}
	So
	$$(1-\lambda)\|\{\alpha_{i}T_{i}x\}_{i}\|\leq(1+\mu)\|\{\beta_{i}R_{i}x\}_{i}\|.$$
	Hence
	$$(1-\lambda)\inf_{\mathbb{J}}\alpha_{i}\|\{T_{i}x\}_{i}\|\leq(1+\mu)\sup_{\mathbb{J}}\beta_{i}\|\{R_{i}x\}_{i}\|.$$
	Thus
	$$\frac{(1-\lambda)\inf_{\mathbb{J}}\alpha_{i}}{(1+\mu)\sup_{\mathbb{J}}\beta_{i}}\|\{T_{i}x\}_{i}\|\leq\|\{R_{i}x\}_{i}\|.$$
	This gives
	\begin{align*}
	A(\frac{(1-\lambda)\inf_{\mathbb{J}}\alpha_{i}}{(1+\mu)\sup_{\mathbb{J}}\beta_{i}})^{2}\|\langle x,x\rangle\|\leq(\frac{(1-\lambda)\inf_{\mathbb{J}}\alpha_{i}}{(1+\mu)\sup_{\mathbb{J}}\beta_{i}})^{2}\|\{T_{i}x\}_{i}\|^{2}\leq\|\{R_{i}x\}_{i}\|^{2}\\\leq(\frac{(1+\lambda)\sup_{\mathbb{J}}\alpha_{i}}{(1-\mu)\inf_{\mathbb{J}}\beta_{i}})^{2}\|\{T_{i}x\}_{i}\|^{2}\leq B(\frac{(1+\lambda)\sup_{\mathbb{J}}\alpha_{i}}{(1-\mu)\inf_{\mathbb{J}}\beta_{i}})^{2}\|\langle x,x\rangle\|	
	\end{align*}
	Hence
	$$A(\frac{(1-\lambda)\inf_{\mathbb{J}}\alpha_{i}}{(1+\mu)\sup_{\mathbb{J}}\beta_{i}})^{2}\|\langle x,x\rangle\|\leq\|\sum_{i\in\mathbb{J}}\langle R_{i}x,R_{i}x\rangle\|\leq B(\frac{(1+\lambda)\sup_{\mathbb{J}}\alpha_{i}}{(1-\mu)\inf_{\mathbb{J}}\beta_{i}})^{2}\|\langle x,x\rangle\|.$$
	Then, $\{R_{i}\}_{i}$ is a $K$-operator frame for $End_{\mathcal{A}}^{\ast}(\mathcal{H})$.	
\end{proof}
\begin{thm}
	Let $\{T_{i}\}_{i\in\mathbb{J}}$ be a $K$-operator frame for $End_{\mathcal{A}}^{\ast}(\mathcal{H})$ with bound $A$ and $B$. Let $\{R_{i}\}_{i\in\mathbb{J}}\subset End_{\mathcal{A}}^{\ast}(\mathcal{H})$ and $\alpha,\beta\geq0$. If $0\leq\alpha+\frac{\beta}{A}<1$ such that
	\begin{equation*}
	\|\sum_{i\in\mathbb{J}}\langle(T_{i}-R_{i})x,(T_{i}-R_{i})x\rangle\|\leq\alpha\|\sum_{i\in\mathbb{J}}\langle T_{i}x,T_{i}x\rangle\|+\beta\|\langle K^{\ast}x,K^{\ast}x\rangle\|, \forall x\in\mathcal{H}.
	\end{equation*}
	Then $\{R_{i}\}_{i\in\mathbb{J}}$ is a $K$-operator frame with frame bounds $A(1-\sqrt{\alpha+\frac{\beta}{A}})^{2}$ and $B(1+\sqrt{\alpha+\frac{\beta}{A}})^{2}$.
\end{thm}
\begin{proof}
	Let $\{T_{i}\}_{i\in\mathbb{J}}$ be a $K$-operator frame for $End_{\mathcal{A}}^{\ast}(\mathcal{H})$  with bound $A$ and $B$. Then for all $x\in\mathcal{H}$, we have
	\begin{align*}
	\|\{T_{i}x\}_{i}\|&\leq\|\{(T_{i}-R_{i})x\}_{i}\|+\|\{R_{i}x\}_{i}\|\\
	&\leq(\alpha\|\sum_{i\in\mathbb{J}}\langle T_{i}x,T_{i}x\rangle\|+\beta\|\langle K^{\ast}x,K^{\ast}x\rangle\|)^{\frac{1}{2}}+\|\sum_{i\in\mathbb{J}}\langle R_{i}x,R_{i}x\rangle\|^{\frac{1}{2}}\\
	&\leq(\alpha\|\sum_{i\in\mathbb{J}}\langle T_{i}x,T_{i}x\rangle\|+\frac{\beta}{A}\|\sum_{i\in\mathbb{J}}\langle T_{i}x,T_{i}x\rangle\|)^{\frac{1}{2}}+\|\sum_{i\in\mathbb{J}}\langle R_{i}x,R_{i}x\rangle\|^{\frac{1}{2}}\\
	&=\sqrt{\alpha+\frac{\beta}{A}}\|\{T_{i}x\}_{i}\|+\|\sum_{i\in\mathbb{J}}\langle R_{i}x,R_{i}x\rangle\|^{\frac{1}{2}}.
	\end{align*}
	Then
	\begin{equation*}
	(1-\sqrt{\alpha+\frac{\beta}{A}})\|\{T_{i}x\}_{i}\|\leq\|\sum_{i\in\mathbb{J}}\langle R_{i}x,R_{i}x\rangle\|^{\frac{1}{2}}.
	\end{equation*}
	This gives
	\begin{equation*}
	A(1-\sqrt{\alpha+\frac{\beta}{A}})^{2}\|\langle K^{\ast}x,K^{\ast}x\rangle\|\leq(1-\sqrt{\alpha+\frac{\beta}{A}})^{2}\|\sum_{i\in\mathbb{J}}\langle T_{i}x,T_{i}x\rangle\|\leq\|\sum_{i\in\mathbb{J}}\langle R_{i}x,R_{i}x\rangle\|.
	\end{equation*}
	Also, we have
	\begin{align*}
	\|\{R_{i}x\}_{i}\|&\leq\|\{(T_{i}-R_{i})x\}_{i}\|+\|\{T_{i}x\}_{i}\|\\
	&\leq\sqrt{\alpha+\frac{\beta}{A}}\|\{T_{i}x\}_{i}\|+\|\{T_{i}x\}_{i}\|\\
	&=(1+\sqrt{\alpha+\frac{\beta}{A}})\|\{T_{i}x\}_{i}\|\\
	&\leq\sqrt{B}(1+\sqrt{\alpha+\frac{\beta}{A}})\|\langle x,x\rangle\|.
	\end{align*}
	Then
	\begin{equation*}
	\|\sum_{i\in\mathbb{J}}\langle R_{i}x,R_{i}x\rangle\|\leq B(1+\sqrt{\alpha+\frac{\beta}{A}})^{2}\|\langle x,x\rangle\|.
	\end{equation*}
	So
	\begin{equation*}
	A(1-\sqrt{\alpha+\frac{\beta}{A}})^{2}\|\langle K^{\ast}x,K^{\ast}x\rangle\|\leq\|\sum_{i\in\mathbb{J}}\langle R_{i}x,R_{i}x\rangle\|\leq B(1+\sqrt{\alpha+\frac{\beta}{A}})^{2}\|\langle x,x\rangle\|.
	\end{equation*}
	Hence $\{R_{i}\}_{i\in\mathbb{J}}$ is a $K$-operator frame with frame bounds $A(1-\sqrt{\alpha+\frac{\beta}{A}})^{2}$ and $B(1+\sqrt{\alpha+\frac{\beta}{A}})^{2}$.
\end{proof}
\begin{cor}
	Let $\{T_{i}\}_{i\in\mathbb{J}}$ be a $K$-operator frame for $End_{\mathcal{A}}^{\ast}(\mathcal{H})$ with bound $A$ and $B$. Let $\{R_{i}\}_{i\in\mathbb{J}}\subset End_{\mathcal{A}}^{\ast}(\mathcal{H})$ and $\alpha\geq0$. If $0\leq\alpha<A$ such that
	\begin{equation*}
	\|\sum_{i\in\mathbb{J}}\langle(T_{i}-R_{i})x,(T_{i}-R_{i})x\rangle\|\leq\alpha\|\langle K^{\ast}x,K^{\ast}x\rangle\|, \forall x\in\mathcal{H}.
	\end{equation*}
	Then $\{R_{i}\}_{i\in\mathbb{J}}$ is a $K$-operator frame with frame bounds $A(1-\sqrt{\frac{\alpha}{A}})^{2}$ and $B(1+\sqrt{\frac{\alpha}{A}})^{2}$.
\end{cor}
\begin{proof}
	Follows in view of the last Theorem.
\end{proof}
\begin{thm}
	Let $\{T_{i}\}_{i\in\mathbb{J}}$ be a $K$-operator frame for $End_{\mathcal{A}}^{\ast}(\mathcal{H})$ with bound $A$ and $B$. Let $\{R_{i}\}_{i\in\mathbb{J}}\in End_{\mathcal{A}}^{\ast}(\mathcal{H})$. If there exists a constant $M>0$, such that for all $x\in\mathcal{H}$, we have
	\begin{equation}\label{eq51}
	\|\sum_{i\in\mathbb{J}}\langle(T_{i}-R_{i})x,(T_{i}-R_{i})x\rangle\|\leq M\min(\|\sum_{i\in\mathbb{J}}\langle T_{i}x,T_{i}x\rangle\|,\|\sum_{i\in\mathbb{J}}\langle R_{i}x,R_{i}x\rangle\|).
	\end{equation}
	Then $\{R_{i}\}_{i\in\mathbb{J}}$ is a $K$-operator frame for $End_{\mathcal{A}}^{\ast}(\mathcal{H})$. The converse is valid for any co-isometry operator $K$.
\end{thm}
\begin{proof}
	We suppose that \eqref{eq51} holds. For evry $x\in\mathcal{H}$, we have 
	\begin{align}\label{eq52}
	\sqrt{A}\|K^{\ast}x\|\leq\|\sum_{i\in\mathbb{J}}\langle T_{i}x,T_{i}x\rangle\|^{\frac{1}{2}}&=\|\{T_{i}x\}_{i\in\mathbb{J}}\|\leq\|\{(T_{i}-R_{i})x\}_{i\mathbb{J}}\|+\|\{R_{i}x\}_{i\in\mathbb{J}}\|\notag\\
	&=\|\sum_{i\in\mathbb{J}}\langle(T_{i}-R_{i})x,(T_{i}-R_{i})x\rangle\|^{\frac{1}{2}}+\|\sum_{i\in\mathbb{J}}\langle R_{i}x,R_{i}x\rangle\|^{\frac{1}{2}}\notag\\
	&\leq\sqrt{M}\|\sum_{i\in\mathbb{J}}\langle R_{i}x,R_{i}x\rangle\|^{\frac{1}{2}}+\|\sum_{i\in\mathbb{J}}\langle R_{i}x,R_{i}x\rangle\|^{\frac{1}{2}}\notag\\
	&=(\sqrt{M}+1)\|\sum_{i\in\mathbb{J}}\langle R_{i}x,R_{i}x\rangle\|^{\frac{1}{2}}.
	\end{align}
	Also we have
	\begin{align}\label{eq53}
	\|\sum_{i\in\mathbb{J}}\langle R_{i}x,R_{i}x\rangle\|^{\frac{1}{2}}&=\|\{R_{i}x\}_{i\in\mathbb{J}}\|\leq\|\{T_{i}x\}_{i\in\mathbb{J}}\|+\|\{(T_{i}-R_{i})x\}_{i\mathbb{J}}\|\notag\\
	&=\|\sum_{i\in\mathbb{J}}\langle T_{i}x,T_{i}x\rangle\|^{\frac{1}{2}}+\|\sum_{i\in\mathbb{J}}\langle(T_{i}-R_{i})x,(T_{i}-R_{i})x\rangle\|^{\frac{1}{2}}\notag\\
	&\leq\|\sum_{i\in\mathbb{J}}\langle T_{i}x,T_{i}x\rangle\|^{\frac{1}{2}}+\sqrt{M}\|\sum_{i\in\mathbb{J}}\langle T_{i}x,T_{i}x\rangle\|^{\frac{1}{2}}\notag\\
	&=(\sqrt{M}+1)\sum_{i\in\mathbb{J}}\langle T_{i}x,T_{i}x\rangle\|^{\frac{1}{2}}\notag\\
	&\leq\sqrt{B}(\sqrt{M}+1)\|x\|.
	\end{align}
	Of \eqref{eq52} and \eqref{eq53} we obtain
	\begin{equation*}
	\frac{A}{(\sqrt{M}+1)^{2}}\|K^{\ast}x\|^{2}\leq\|\sum_{i\in\mathbb{J}}\langle R_{i}x,R_{i}x\rangle\|\leq B(\sqrt{M}+1)^{2}\|x\|^{2}.
	\end{equation*}
	so we have $\{R_{i}\}_{i\in\mathbb{J}}$ is a $K$-operator frame for $End_{\mathcal{A}}^{\ast}(\mathcal{H})$.\\
	For the converse suppose that $\{R_{i}\}_{i\in\mathbb{J}}\in End_{\mathcal{A}}^{\ast}(\mathcal{H})$ be an operator frame with bound $C$ and $D$, and $K$ is a co-isometry operator on $\mathcal{H}$, i.e., $\|K^{\ast}x\|=\|x\|, \forall x\in\mathcal{H}$. Then for any $x\in\mathcal{H}$, we have
	\begin{align*}
	\|\sum_{i\in\mathbb{J}}\langle(T_{i}-R_{i})x,(T_{i}-R_{i})x\rangle\|^{\frac{1}{2}}&=\|\{(T_{i}-R_{i})x\}_{i\in\mathbb{J}}\|\leq\|\{T_{i}x\}_{i\in\mathbb{J}}\|+\|\{R_{i}x\}_{i\in\mathbb{J}}\|\\
	&=\|\sum_{i\in\mathbb{J}}\langle T_{i}x,T_{i}x\rangle\|^{\frac{1}{2}}+\|\sum_{i\in\mathbb{J}}\langle R_{i}x,R_{i}x\rangle\|^{\frac{1}{2}}\\
	&\leq\|\sum_{i\in\mathbb{J}}\langle T_{i}x,T_{i}x\rangle\|^{\frac{1}{2}}+\sqrt{D}\|x\|\\
	&=\|\sum_{i\in\mathbb{J}}\langle T_{i}x,T_{i}x\rangle\|^{\frac{1}{2}}+\sqrt{D}\|K^{\ast}x\|\\
	&\leq\|\sum_{i\in\mathbb{J}}\langle T_{i}x,T_{i}x\rangle\|^{\frac{1}{2}}+\sqrt{\frac{D}{A}}\|\sum_{i\in\mathbb{J}}\langle T_{i}x,T_{i}x\rangle\|^{\frac{1}{2}}\\
	&=(1+\sqrt{\frac{D}{A}})\|\sum_{i\in\mathbb{J}}\langle T_{i}x,T_{i}x\rangle\|^{\frac{1}{2}}
	\end{align*}
	Similary we can obtain
	\begin{equation*}
	\|\sum_{i\in\mathbb{J}}\langle(T_{i}-R_{i})x,(T_{i}-R_{i})x\rangle\|^{\frac{1}{2}}\leq(1+\sqrt{\frac{B}{C}})\|\sum_{i\in\mathbb{J}}\langle R_{i}x,R_{i}x\rangle\|^{\frac{1}{2}}
	\end{equation*}
	Let $M=\min\{1+\sqrt{\frac{D}{A}},1+\sqrt{\frac{B}{C}}\}$, then \eqref{eq51} holds.\\
\end{proof}
\begin{thm}
	Let $K\in End_{\mathcal{A}}^{\ast}(\mathcal{H})$. For $n=1,2,..,k$, let $\{T_{n,i}\}_{i}\subset End_{\mathcal{A}}^{\ast}(\mathcal{H})$ be $K$-operator frame for $End_{\mathcal{A}}^{\ast}(\mathcal{H})$ with bound $A_{n}$ and $B_{n}$, $\{\alpha_{n}\}_{n=1}^{k}$ be any scalars. If there exists a constant $\lambda>0$ and $p\in\{1,2,...,k\}$ such that
	\begin{equation*}
	\lambda\|\{T_{p,i}\}_{i}\|\leq\|\{\sum_{n=1}^{k}\alpha_{n}T_{n,i}(x)\}_{i}\|, x\in\mathcal{H}.
	\end{equation*}
	Then $\{\sum_{n=1}^{k}\alpha_{n}T_{n,i}\}_{i}$ is a $K$-operator frame for $End_{\mathcal{A}}^{\ast}(\mathcal{H})$. The converse is valid for any co-isometry operator $K$.
\end{thm}
\begin{proof}
	For any $x\in\mathcal{H}$, we have
	\begin{align*}
	\sqrt{A_{p}}\lambda\|\langle K^{\ast}x,K^{\ast}x\rangle\|^{\frac{1}{2}}&\leq\lambda\|\{T_{p,i}x\}_{i}\|\\
	&\leq\|\{\sum_{n=1}^{k}\alpha_{n}T_{n,i}(x)\}_{i}\|\\
	&\leq\sum_{n=1}^{k}|\alpha_{n}|\|\{T_{n,i}x\}_{i}\|\\
	&\leq\max_{1\leq n\leq k}|\alpha_{n}|\sum_{n=1}^{k}\|\{T_{n,i}x\}_{i}\|\\
	&\leq\max_{1\leq n\leq k}|\alpha_{n}|(\sum_{n=1}^{k}\sqrt{B_{n}})\|\langle x,x\rangle\|^{\frac{1}{2}}.
	\end{align*}
	Then forall $x\in\mathcal{H}$, we have
	\begin{equation*}
	\sqrt{A_{p}}\lambda\|\langle K^{\ast}x,K^{\ast}x\rangle\|^{\frac{1}{2}}\leq\|\{\sum_{n=1}^{k}\alpha_{n}T_{n,i}(x)\}_{i}\|\leq\max_{1\leq n\leq k}|\alpha_{n}|(\sum_{n=1}^{k}\sqrt{B_{n}})\|\langle x,x\rangle\|^{\frac{1}{2}}.
	\end{equation*}
	So forall $x\in\mathcal{H}$,
	\begin{equation*}
	A_{p}\lambda^{2}\|\langle K^{\ast}x,K^{\ast}x\rangle\|\leq\|\{\sum_{n=1}^{k}\alpha_{n}T_{n,i}(x)\}_{i}\|^{2}\leq(\max_{1\leq n\leq k}|\alpha_{n}|)^{2}(\sum_{n=1}^{k}\sqrt{B_{n}})^{2}\|\langle x,x\rangle\|.
	\end{equation*}
	Hence $\{\sum_{n=1}^{k}\alpha_{n}T_{n,i}\}_{i}$ is a $K$-operator frame for $End_{\mathcal{A}}^{\ast}(\mathcal{H})$.\\
	Conversly, let $K$ is a co-isometry operator on $\mathcal{H}$, let $\{\sum_{n=1}^{k}\alpha_{n}T_{n,i}\}_{i}$ is a $K$-operator frame for $End_{\mathcal{A}}^{\ast}(\mathcal{H})$ with bounds $A$, $B$ and let for any $p\in\{1,2,...,k\}$, $\{T_{p,i}\}_{i}$ be an  a $K$-operator frame for $End_{\mathcal{A}}^{\ast}(\mathcal{H})$ with bounds $A_{p}$ and $B_{p}$. Then, for any $x\in\mathcal{H}$, $p\in\{1,2,...,k\}$, we have
	\begin{equation*}
	A_{p}\|\langle K^{\ast}x,K^{\ast}x\rangle\|\leq\|\{T_{p,i}x\}_{i}\|^{2}\leq B_{p}\|\langle x,x\rangle\|.
	\end{equation*}
	This gives
	\begin{equation*}
	\frac{1}{B_{p}}\|\{T_{p,i}x\}_{i}\|^{2}\leq\|\langle x,x\rangle\|, x\in\mathcal{H}.
	\end{equation*}
	Also, we have
	\begin{equation*}
	A\|\langle K^{\ast}x,K^{\ast}x\rangle\|\leq\|\{\sum_{n=1}^{k}\alpha_{n}T_{n,i}(x)\}_{i}\|^{2}\leq B\|\langle x,x\rangle\|, x\in\mathcal{H}.
	\end{equation*}
	So,
	\begin{equation*}
	\|\langle x,x\rangle\|=\|\langle K^{\ast}x,K^{\ast}x\rangle\|\leq\frac{1}{A}\|\{\sum_{n=1}^{k}\alpha_{n}T_{n,i}(x)\}_{i}\|^{2}, x\in\mathcal{H}.
	\end{equation*}
	Hence
	\begin{equation*}
	\frac{A}{B_{p}}\|\{T_{p,i}x\}_{i}\|^{2}\leq\|\{\sum_{n=1}^{k}\alpha_{n}T_{n,i}(x)\}_{i}\|^{2}, x\in\mathcal{H}.
	\end{equation*}
	Then for $\lambda=\frac{A}{B_{p}}$, we have
	\begin{equation*}
	\lambda\|\{T_{p,i}x\}_{i}\|^{2}\leq\|\{\sum_{n=1}^{k}\alpha_{n}T_{n,i}(x)\}_{i}\|^{2}, x\in\mathcal{H}.
	\end{equation*}	
\end{proof}
\begin{thm}
	Let $K\in End_{\mathcal{A}}^{\ast}(\mathcal{H})$. For $n=1,2,..,k$, let $\{T_{n,i}\}_{i}\subset End_{\mathcal{A}}^{\ast}(\mathcal{H})$ be $K$-operator frame for $End_{\mathcal{A}}^{\ast}(\mathcal{H})$ with bound $A_{n}$ and $B_{n}$, $\{R_{n,i}\}_{i}\subset End_{\mathcal{A}}^{\ast}(\mathcal{H})$ be any sequence. Let $L:\ell^{2}(\mathcal{H})\rightarrow\ell^{2}(\mathcal{H})$ be a bounded linear operator such that $L(\{\sum_{n=1}^{k}R_{n,i}(x)\}_{i})=\{T_{p,i}(x)\}_{i}$, for some $p\in\{1,2,...,k\}$. If there exists a constant $\lambda>0$ such that\begin{equation*}
	\|\sum_{i}\langle(T_{n,i}-R_{n,i})x,(T_{n,i}-R_{n,i})x\rangle\|\leq\lambda\|\sum_{i}\langle T_{n,i}x,T_{n,i}x\rangle\|, x\in\mathcal{H}, n=1,2,..,k.
	\end{equation*}
	Then $\{\sum_{n=1}^{k}R_{n,i}\}_{i}$ is a $K$-operator frame for $End_{\mathcal{A}}^{\ast}(\mathcal{H})$.
\end{thm}
\begin{proof}
	For any $x\in\mathcal{H}$, we have
	\begin{align*}
	\|\{\sum_{n=1}^{k}R_{n,i}(x)\}_{i}\|&\leq\sum_{n=1}^{k}\|\{R_{n,i}(x)\}_{i}\|\\
	&\leq\sum_{n=1}^{k}(\|\{(T_{n,i}-R_{n,i})x\}_{i}\|+\|\{T_{n,i}(x)\}_{i}\|)\\
	&\leq\sum_{n=1}^{k}(\sqrt{\lambda}\|\{T_{n,i}(x)\}_{i}\|+\|\{T_{n,i}(x)\}_{i}\|)\\
	&=(1+\sqrt{\lambda})\sum_{n=1}^{k}\|\{T_{n,i}(x)\}_{i}\|\\
	&\leq(1+\sqrt{\lambda})(\sum_{n=1}^{k}\sqrt{B_{n}})\|\langle x,x\rangle\|^{\frac{1}{2}}.
	\end{align*}
	Also, for each $x\in\mathcal{H}$, we have
	\begin{equation*}
	\|L(\{\sum_{n=1}^{k}R_{n,i}(x)\}_{i})\|=\|\{T_{p,i}(x)\}_{i}\|
	\end{equation*}
	Therefore, we get
	\begin{align*}
	\sqrt{A_{p}}\|\langle K^{\ast}x,K^{\ast}x\rangle\|^{\frac{1}{2}}\leq\|\{T_{p,i}x\}_{i}\|&=\|L(\{\sum_{n=1}^{k}R_{n,i}(x)\}_{i})\|\\
	&\leq\|L\|\|\{\sum_{n=1}^{k}R_{n,i}(x)\}_{i}\|, x\in\mathcal{H}
	\end{align*}
	This gives
	\begin{equation*}
	\frac{\sqrt{A_{p}}}{\|L\|}\|\langle K^{\ast}x,K^{\ast}x\rangle\|^{\frac{1}{2}}\leq\|\{\sum_{n=1}^{k}R_{n,i}(x)\}_{i}\|, x\in\mathcal{H}
	\end{equation*}
	Then
	\begin{equation*}
	\frac{\sqrt{A_{p}}}{\|L\|}\|\langle K^{\ast}x,K^{\ast}x\rangle\|^{\frac{1}{2}}\leq\|\{\sum_{n=1}^{k}R_{n,i}(x)\}_{i}\|\leq(1+\sqrt{\lambda})(\sum_{n=1}^{k}\sqrt{B_{n}})\|\langle x,x\rangle\|^{\frac{1}{2}},x\in\mathcal{H}.
	\end{equation*}
	So
	\begin{equation*}
	\frac{A_{p}}{\|L\|^{2}}\|\langle K^{\ast}x,K^{\ast}x\rangle\|\leq\|\{\sum_{n=1}^{k}R_{n,i}(x)\}_{i}\|^{2}\leq(1+\sqrt{\lambda})^{2}(\sum_{n=1}^{k}\sqrt{B_{n}})^{2}\|\langle x,x\rangle\|,x\in\mathcal{H}.
	\end{equation*}
	Hence $\{\sum_{n=1}^{k}R_{n,i}\}_{i}$ is a $K$-operator frame for $End_{\mathcal{A}}^{\ast}(\mathcal{H})$.
\end{proof} 

\subsection*{Acknowledgments}  The authors would like to thank the referee for useful and helpful comments and
suggestions.


\begin{thebibliography}{9}

\bibitem{Asg} M. S. Asgari, A. Khorsavi, \emph{Frames and Bases of subspaces in Hilbert spaces},J. Math. Anal. Appl. 308 (2005),
541-553.
\bibitem{Ara} L. Aramba\v{s}i\'{c}, \emph{On frames for countably generated Hilbert $\mathcal{C}^{\ast}$-modules}, Proc. Amer. Math. Soc. 135 (2007) 469-478.
\bibitem{Deh} A. Alijani,M. Dehghan, \emph{$\ast$-frames in Hilbert $\mathcal{C}^{\ast}$modules},U. P. B. Sci. Bull. Series A 2011.
\bibitem{Ali} A. Alijani, Generalized frames with C*-valued bounds and their operator duals, {\it Filomat}, {\bf 29}(7) (2015), 1469-1479.
\bibitem{Kab} N. Bounader,S. Kabbaj, \emph{$\ast$-g-frames in Hilbert $C^{\ast}$-modules},J. Math. Comput. Sci. 4, No. 2, 246-256,(2014).
\bibitem{Con} J. B. Conway ,\emph{A Course In Operator Theory},AMS,V.21,2000.
\bibitem{Li} C. Y. LI, H. X. CAO, Operator frames for $B(\mathcal{H})$, in: T. Qian, M. I. Vai, X. Yuesheng (eds.), Wavelet
Analysis and Applications, Applications of Numerical Harmonic Analysis, 67-82, Springer, Berlin
(2006).		
\bibitem{Eld} O. Christensen, Y. C. Eldar, \emph{Oblique dual frames and shift-invariant spaces},Appl. Comput. Harmon. Anal.
17 (2004), 48-68.
\bibitem{Chr} O. Christensen, \emph{An Introduction to Frames and Riesz bases}, Brikhouser, 2016.
\bibitem{Duf} R. J. Duffin, A. C. Schaeffer, \emph{A class of nonharmonic fourier series},Trans. Amer. Math. Soc. 72 (1952),
341-366.
\bibitem{Dav} F. R. Davidson, \emph{$\mathcal{C}^{\ast}$-algebra by example},Fields Ins. Monog. 1996.
\bibitem{Lar1} M. Frank, D. R. Larson, \emph{$\mathcal{A}$-module frame concept for Hilbert $\mathcal{C}^{\ast}$-modules},functinal and harmonic analysis of
wavelets, Contempt. Math. 247 (2000), 207-233.
\bibitem{Gab} D. Gabor, \emph{Theory of communications},J. Elec. Eng. 93 (1946), 429-457.
\bibitem{Kap} I. Kaplansky, \emph{Modules over operator algebras},Amer. J. Math. 75 (1953), 839-858.
\bibitem{Kho1} A. Khosravi,B. Khosravi, \emph{Frames and bases in tensor products of Hilbert $\mathcal{C}^{\ast}$-modules},Proc.Indian Acad.Sci.,Math.Sci.117,no.1, 1-12,2007.
\bibitem{Kho2} A. Khorsavi, B. Khorsavi, \emph{Fusion frames and g-frames in Hilbert $\mathcal{C}^{\ast}$-modules},Int. J.Wavelet, Multiresolution
and Information Processing 6 (2008), 433-446.
\bibitem{Oga} S. Li, H. Ogawa, \emph{Pseudoframes for subspaces with applications},J. Fourier Anal. Appl. 10 (2004), 409-431.
\bibitem{Pas} W. Paschke, \emph{Inner product modules over $B^{\ast}$-algebras}, Trans. Amer. Math. Soc., (182)(1973), 443-468.
\bibitem{Ros2} M. Rossafi and S. Kabbaj, \emph{$\ast$-K-g-frames in Hilbert $C^{\ast}$-modules}, Journal of
Linear and Topological Algebra Vol. 07, No. 01, 2018, 63- 71.
\bibitem{Ros} M. Rossafi and S. Kabbaj, \emph{K-operator Frame for $End_{\mathcal{A}}^{\ast}(\mathcal{H})$}, Asia Mathematika Volume 2, Issue 2, (2018), 52-60.
\bibitem{She} C. Shekhar, S. K. Kaushik \emph{Frames for $B(\mathcal{H})$}, Operators and Matrices, V.11, nÂ°1, 181-198, (2017).
\bibitem{Zha} L. C. Zhang, \emph{The factor decomposition theorem of bounded generalized inverse modules and their topological continuity}, J. Acta Math. Sin., 23 (2007), 1413-1418.

\end{thebibliography}
\end{document}